\documentclass[final]{siamltex}
\usepackage{amssymb,amsmath}
 \usepackage[latin1]{inputenc}
 \usepackage[OT1]{fontenc}
 \def\kk{\mathbb{K}}
 \def\II{\mathbb{I}}
 
 \def\NN{\mathbb{N}}
 
 \def\Tl{\tilde{T}}
 \def\Tv{\mathbf{T}}
 \def\ML{\mathcal{L}}
 \def\ab{\mathbf{a}}

 \def\hb{\mathbf{h}}

 \def\<{\langle}
 \def\>{\rangle}
 \def\lcm{\mathrm{lcm}}
 \def\Oc{\mathcal{O}}
 \def\Osft{\tilde{\mathcal{O}}}
 \newcommand{\Unit}[1]{\mathfrak{U}_{#1}}
 \newcommand{\ov}[1]{\overline{#1}}
 \newcommand{\ud}[1]{\underline{#1}}
\title{Toeplitz and Toeplitz-block-Toeplitz matrices\\and their correlation with syzygies of polynomials}

\author{
Houssam Khalil\thanks{Institut Camille Jordan, 43 boulevard du 11 novembre 1918, 69622 Villeurbanne cedex France ({\tt khalil@math.univ-lyon1.fr}).}\and Bernard Mourrain\thanks{INRIA, GALAAD team, 2004 route des Lucioles, BP 93, 06902 Sophia Antipolis Cedex, France({\tt mourrain@sophia.inria.fr}).}\and Michelle Schatzman\thanks{Institut Camille Jordan, 43 boulevard du 11 novembre 1918, 69622 Villeurbanne cedex France ({\tt schatz@math.univ-lyon1.fr}).}
}
\begin{document}

\maketitle

\begin{abstract}
In this paper, we re-investigate the resolution of Toeplitz systems $T\, u =g$, from a new point of view, by correlating the solution of such problems with syzygies of polynomials or moving lines. We show an explicit connection between the generators of a Toeplitz matrix and the generators of the corresponding module of syzygies. We show that this module is generated by two elements of degree $n$ and the solution of $T\,u=g$ can be reinterpreted as the remainder of an explicit vector depending on $g$, by these two generators.

This approach extends naturally to multivariate problems and we describe 
for Toeplitz-block-Toeplitz matrices, the structure of the corresponding generators. 
\end{abstract}

\section{Introduction}
Structured matrices appear in various domains, such as scientific computing,
signal processing, \dots They usually express, in a linearize way, a problem
which depends on less parameters than the number of entries of the
corresponding matrix. An important area of research is devoted to the
development of methods for the treatment of such matrices, which depend on
the actual parameters involved in these matrices.

Among well-known structured matrices, Toeplitz and Hankel structures have
been intensively studied \cite{MR782105, MR1355506}. Nearly optimal
algorithms are known for the multiplication or the resolution of linear
systems, for such structure. Namely, if $A$ is a Toeplitz matrix of size $n$,
multiplying it by a vector or solving a linear system with $A$ requires
$\Osft(n)$ arithmetic operations (where $\Osft(n)=\Oc(n \log^{c}(n))$ for
some $c>0$) \cite{LDRBA80, MR1871324}. Such algorithms are called super-fast,
in opposition with fast algorithms requiring $\Oc(n^{2})$ arithmetic
operations.

The fundamental ingredients in these algorithms are the so-called generators
\cite{MR1355506}, encoding the minimal information stored in these matrices,
and on which the matrix transformations are translated.  The correlation with
other types of structured matrices has also been well developed in the
literature \cite{MR1843842,MR1755557}, allowing to treat so efficiently other
structures such as Vandermonde or Cauchy-like structures.

Such problems are strongly connected to polynomial problems 
\cite{LDRFuh96b, MR1289412}. For instance, the product of a Toeplitz matrix by a
vector can be deduced from the product of two univariate polynomials, and
thus can be computed efficiently by evaluation-interpolation techniques,
based on FFT. The inverse of a Hankel or Toeplitz matrix is connected to the
Bezoutian of the polynomials associated to their generators.

However, most of these methods involve univariate polynomials. So far, few
investigations have been pursued for the treatment of multilevel structured
matrices \cite{LDRTyr85}, related to multivariate problems. Such linear
systems appear for instance in resultant or in residue
constructions, in normal form computations, or more generally in
multivariate polynomial algebra. We refer to \cite{MR1762401} for a
general description of such correlations between multi-structured matrices 
and multivariate polynomials. Surprisingly, they also appear in numerical
scheme and preconditionners.
A main challenge here is to devise super-fast algorithms of complexity
$\Osft(n)$ for the resolution of multi-structured systems of size $n$.

In this paper, we consider block-Toeplitz matrices, where each block is a
Toeplitz matrix. Such a structure, which is the first step to multi-level
structures, is involved in many bivariate problems, or in numerical
linear problems.We re-investigate first the resolution of Toeplitz systems
$T\, u =g$,
from a new point of view, by correlating the solution of such problems with
syzygies of polynomials or moving lines. We show an explicit connection
between the generators of a Toeplitz matrix and the generators of the
corresponding module of syzygies. 
We show that this module is generated by two elements of degree $n$ and the
solution of $T\,u=g$ can be reinterpreted as the remainder of an explicit
vector depending on $g$, by these two generators.

This approach extends naturally to multivariate problems and we describe 
for Toeplitz-block-Toeplitz matrices, the structure of the corresponding generators. 
In particular, we show the known result that the module of syzygies of
$k$ non-zero bivariate polynomials is free of rank $k-1$, by a new elementary
proof.

Exploiting the properties of moving lines associated to Toeplitz matrices, we give a new point of view to resolve a Toeplitz-block-Toeplitz system.

In the next section we studie the scalar Toeplitz case. In the chapter 3 we consider the Toeplitz-block-Toeplitz case.

Let $R=\kk[x]$. For $n \in \NN$, we denote by $\kk[x]_{n}$ the vector space of polynomials of
degree $\le n$. Let $L=\kk[x,x^{-1}]$ be the set of Laurent
polynomials in the variable $x$. For any polynomial $p=\sum_{i=-m}^{n}
p_{i}\, x^{i} \in L$, we denote by $p^{+}$ the sum of terms with positive
exponents: $p^{+}=\sum_{i=0}^{n} p_{i}\, x^{i}$ and by $p^{-}$, the sum of terms
with strictly negative exponents: $p^{-}=\sum_{i=-m}^{-1} p_{i}\, x^{i}$.
We have $p=p^{+} +p^{-}$.

For $n\in \NN$, we denote by $\Unit{n}=\{\omega; \omega^{n}=1\}$ the set of roots of unity of
order $n$.

\section{Univariate case}
We begin by the univariate case and the following problem:
\begin{problem} Given a Toeplitz matrix $T=(t_{i-j})_{i,j=0}^{n-1}\in
\kk^{n\times n}$  ($T=(T_{ij})_{i,j=0}^{n-1}$ with $T_{ij}=t_{i-j}$) of size
$n$ and $g=(g_0,\dots,g_{n-1}) \in \kk^{n}$, find $u=(u_0,\dots,u_{n-1}) \in
\kk^{n}$ such that 
\begin{equation}\label{pb:toep}
T\,u=g.
\end{equation}
\end{problem}
% This problem is clearly equivalent to the following one:
% \begin{problem} Given a Hankel matrix $H=$\\$(h_{i+j})_{i,j=0}^{n-1}\in
% \kk^{n\times n}$ of size
% $n$ and $g=(g_0,\dots,g_{n-1}) \in \kk^{n}$, find $w=(w_0,\dots,w_{n-1}) \in
% \kk^{n}$ such that 
% \begin{equation}\label{pb:hkl}
% H\,w=g
% \end{equation}
% \end{problem}
% We simply take $T=H\,J$ and $w=J\, u$, where $J$ is the $n\times n$ matrix, with $1$ on
% the antidiagonal and $0$ else where.

Let $E=\{1,\dots,x^{n-1}\}$, and $\Pi_{E}$ be the projection of
$R$ on the vector space generated by $E$, along $\<x^{n},x^{n+1},\ldots\>$.   
\begin{definition}
We define the following polynomials: 
\begin{itemize}
 \item $T(x)=\displaystyle\sum _{i=-n+1}^{n-1}t_ix^i,$
 \item $\tilde{T}(x)=\displaystyle\sum_{i=0}^{2n-1}\tilde{t}_ix^i$ with
$\tilde{t}_i=\left\{
\begin{array}{ll} t_i&\textrm{ if } i< n\\ 
t_{i-2n}&\textrm{ if } i\ge n
\end{array}\right.$,
 \item  $u(x)=\displaystyle\sum_{i=0}^{n-1}u_ix^i,\:g(x)=\sum_{i=0}^{n-1}g_ix^i$.
\end{itemize}
\end{definition}

\noindent{}Notice that $\tilde{T}= T^{+} + x^{2\,n}\, T^{-}$ and
$T(w)=\tilde{T}(w)$ if $w \in \Unit{2\,n}$. We also have (see \cite{MR1762401})
$$
T\,u=g\Leftrightarrow\Pi_{E}(T(x)u(x))=g(x).
$$
For any polynomial $u \in \kk[x]$ of degree $d$, we denote it as $u(x)=
\ud{u}(x) + x^{n} \ov{u}(x)$ with $\deg(\ud{u})\le n-1$ and $\deg(\ov{u}) \le
d-n$ if $d\ge n$ and $\ov{u}=0$ otherwise.
Then, we have 
\begin{eqnarray}
T(x)\, u(x) &=& T(x) \ud{u}(x) + T(x) x^{n} \ov{u}(x) \nonumber \\
& = & \Pi_{E}(T(x) \ud{u}(x)) + \Pi_{E}(T(x) x^{n} \ov{u}(x)) \nonumber \\
&   & + (\alpha_{-n+1}x^{-n+1}+\dots+\alpha_{-1}x^{-1}) \nonumber \\
&   & + (\alpha_{n}x^n+\dots+\alpha_{n+m}x^{n+m}) \nonumber \\ \label{eq:TU1}
 &=& \Pi_{E}(T(x) \ud{u}(x)) + \Pi_{E}(T(x) x^{n} \ov{u}(x)) \nonumber \\
& & + x^{-n+1} A(x) + x^{n} B(x),
\end{eqnarray}
with $m= \max(n-2, d-1),$
\begin{eqnarray}\label{eq:TU2}
&& A(x) = \alpha_{-n+1}+\dots+\alpha_{-1}x^{n-2}, \nonumber \\
&& B(x) = \alpha_{n}+\dots+\alpha_{n+m}x^{m}.
\end{eqnarray}
See \cite{MR1762401} for more details, on the correlation between structured
matrices and (multivariate) polynomials.

\subsection{Moving lines and Toeplitz matrices}
We consider here another problem, related to interesting questions in
Effective Algebraic Geometry.
\begin{problem}\label{pb:mvlines}
Given three polynomials $a, b, c \in R$  respectively of degree $<l, <m, <n$,
find three polynomials  $p, q, r \in R$ 
of degree $< \nu-l, <\nu-m, <\nu-n$, such that
\begin{equation} \label{eq:mvlines}
a(x)\, p(x) + b(x)\, q(x) + c(x)\, r(x) =0.
\end{equation}
\end{problem}
We denote by $\ML(a,b,c)$  the set of $(p,q,r)\in\kk[x]^{3}$ which are solutions of \eqref{eq:mvlines}.
It is a $\kk[x]$-module of $\kk[x]^{3}$. The solutions of the problem \eqref{pb:mvlines} are $\ML(a,b,c) \cap
\kk[x]_{\nu-l-1}\times\kk[x]_{\nu-m-1}\times\kk[x]_{\nu-n-1}$. 

Given a new polynomial $d(x)\in \kk[x]$, we denote by $\ML(a,b,c;d)$ the set of
$(p,q,r)\in\kk[x]^{3}$ such that 
$$ 
a(x)\, p(x) + b(x)\, q(x) + c(x)\, r(x) = d(x).
$$
\begin{theorem}
For any non-zero vector of polynomials $(a,b,c)\in \kk[x]^{3}$, the
$\kk[x]$-module $\ML(a,b,c)$ is free of rank $2$.
\end{theorem}
\begin{proof}
By the Hilbert's theorem, the ideal $I$ generated by $(a,b,c)$ has a free
resolution of length at most $1$, that is of the form:
$$
0\rightarrow\kk[x]^p\rightarrow \kk[x]^3\rightarrow \kk[x] \rightarrow \kk[x]/I \rightarrow 0.
$$ 
As $I\neq 0$, for dimensional reasons, we must have $p=2$.
\end{proof}
\begin{definition} A $\mu$-base of $\ML(a,b,c)$ is a basis $(p,q,r)$,
$(p',q',r')$ of $\ML(a,b,c)$, with $(p,q,r)$ of minimal degree $\mu$.
\end{definition}

Notice if $\mu_1$ is the smallest degree of a generator and $\mu_2$ the degree of the second generator $(p',q',r')$, we have $d=max(\deg(a),\deg(b),\deg(c))=\mu_1+\mu_2$. Indeed, we have
\begin{eqnarray*}
\lefteqn{0 \rightarrow \kk[x]_{\nu-d-\mu_1} \oplus \kk[x]_{\nu-d-\mu_2} \rightarrow} \\
& &\kk[x]_{\nu-d}^3\rightarrow \kk[x]_{\nu} \rightarrow \kk[x]_{\nu}/(a,b,c)_{\nu} \rightarrow 0,
\end{eqnarray*}
for $\nu >> 0$. As the alternate sum of the dimension of the $\kk$-vector spaces is zero and $\kk[x]_{\nu}/(a,b,c)_{\nu}$ is $0$ for $\nu >> 0$, we have
\begin{eqnarray*}
0 & = & 3\,(d-\nu-1)  +\nu -\mu_1- d +1 + \nu -\mu_2 -d +1 + \nu +1\\
& = & d -\mu_1 -\mu_2.
\end{eqnarray*}
For $\ML(\Tl(x), x^{n}, x^{2n}-1)$, we have $\mu_1+\mu_2=2\,n$. 
We are going to show now that in fact $\mu_1=\mu_2=n$: 
\begin{proposition}\label{prop:ML}
The $\kk[x]$-module $\ML(\Tl(x), x^{n}, x^{2n}-1)$ has a $n$-basis.
\end{proposition}
\begin{proof}
Consider the map
\begin{eqnarray}\label{syzfunction1}
\kk[x]_{n-1}^3 &\rightarrow & \kk[x]_{3n-1}\\
(p(x),q(x),r(x))  & \mapsto & \Tl(x) p(x) + x^{n} q(x) +(x^{2n}-1) r(x)\nonumber
\end{eqnarray}
which $3n \times 3n$ matrix is of the form 
\begin{equation}\label{form:S}
S:= \left(
\begin{array}{c|c|c}
T_{0} &\mathbf{0}  &  -\II_{n} \\ 
T_{1} & \II_{n}    &  \mathbf{0} \\
T_{2} & \mathbf{0} &  \ \, \II_{n}  \\ 
\end{array}
\right).
\end{equation}
where $T_{0}, T_{1}, T_{2}$ are the coefficient 
matrices of $(\Tl(x)$, $x\, \Tl(x)$, $\ldots,$ $x^{n}\Tl(x))$, respectively for the
list of monomials $(1,\ldots,x^{n-1})$, $(x^{n},\ldots,x^{2n-1})$, 
$(x^{2n},\ldots, x^{3n-1})$. Notice in particular that
$T= T_{0}+T_{2}$

Reducing the first rows of $(T_{0}| \mathbf{0} | -\II_{n})$ by the
last rows $(T_{2}| \mathbf{0} | \II_{n})$, we replace it by the block
$(T_{0}+T_{2}| \mathbf{0} | \mathbf{0})$, without
changing the rank of $S$. As $T=T_{0}+T_{2}$ is invertible, this shows that the matrix
$S$ is of rank $3n$. Therefore, there is no syzygies in degree $n-1$.
As the sum $2n=\mu_1+\mu_2$ and $\mu_{1}\le n,\mu_{2}\le n$ where
$\mu_1,\mu_2$ are the smallest degree of a pair of  generators of
$\ML(\Tl(x), x^{n}, x^{2n}-1)$ of degree $\le n$, we have
$\mu_1=\mu_2=n$. Thus there exist two linearly independent syzygies
$(u_1,v_1,w_1)$, $(u_2,v_2,w_2)$ of degree $n$, which generate $\ML(\Tl(x), x^{n},
x^{2n}-1)$.
\end{proof}

A similar result can also be found in \cite{MR1871324}, but the proof  much
longer than this one, is based on interpolation techniques and explicit computations. 
Let us now describe how to construct explicitly two generators of
$\ML(\Tl(x), x^{n}, x^{2n}-1)$ of degree $n$ (see also \cite{MR1871324}). 

As $\Tl(x)$ is of degree $\le 2\,n -1$ and the map \eqref{syzfunction1} is a surjective function, there exists $(u,v,w) \in \kk[x]_{n-1}^3$ such that 
\begin{equation}\label{base1}
\Tl(x) u(x) + x^n v(x) + (x^{2\,n}-1)\, w = \Tl(x) x^n,
\end{equation}
we deduce that $(u_1,v_1,w_1)=(x^n-u, -v, -w) \in \ML(\Tl(x), x^{n}, x^{2n}-1)$.

As there exists $(u',v',w') \in \kk[x]_{n-1}^3$ such that 
\begin{equation}\label{base2}
\Tl(x) u'(x) + x^n v'(x) + (x^{2\,n}-1)\, w' =1 = x^n\, x^n - (x^{2\,n}-1)
\end{equation}
we deduce that $(u_2,v_2,w_2)=(-u',x^n -v', -w' - 1) \in \ML(\Tl(x), x^{n}, x^{2n}-1)$.

Now, the vectors $(u_1,v_1,w_1)$, $(u_2,v_2,w_2)$ of $\ML( \Tl(x),x^{n},x^{2n}-1)$
are linearly independent since by construction, the coefficient vectors of
$x^{n}$ in $(u_1,v_1,w_1)$ and $(u_2,v_2,w_2)$ are respectively $(1,0,0)$ and $(0,1,0)$.

\begin{proposition}\label{division}
The vector $u$ is solution of \eqref{pb:toep}  if and only if  there exist
$v(x) \in \kk[x]_{n-1}, w(x) \in \kk[x]_{n-1}$ such that 
$$
(u(x), v(x), w(x)) \in \ML(\tilde{T}(x), x^{n}, x^{2n}-1; g(x) )
$$
\end{proposition}
\begin{proof}
The vector $u$ is solution of \eqref{pb:toep} if and only if we have 
$$ 
\Pi_{E}(T(x)u(x))=g(x).
$$
As $u(x)$ is of degree $\le n-1$, we deduce from \eqref{eq:TU1} and
\eqref{eq:TU2} that there exist polynomial $A(x) \in \kk[x]_{n-2}$ and
$B(x) \in \kk[x]_{n-1}$ such that 
$$ 
T(x)u(x) - x^{-n+1} A(x) - x^{n} B(x) = g(x).
$$
By evaluation at the roots $\omega \in \Unit{2n}$, and since $\omega^{-n}
= \omega^{n}$ and $\Tl(\omega)=T(\omega)$ for $\omega\in \Unit{n}$, we have
$$ 
\Tl(\omega) u(\omega) + \omega^{n} v(\omega) = g(\omega),
\forall \omega \in \Unit{2n}(\omega),
$$
with $v(x)= -x\, A(x)-B(x)$ of degree $\le n-1$.
We deduce that there exists
$w(x)\in\kk[x]$ such that
$$ 
\Tl(x) u(x) + x^{n} v(x) + (x^{2n}-1) w(x)= g(x).
$$
Notice that $w(x)$ is of degree $\le n-1$, because $(x^{2n}-1)\, w(x)$ is of
degree $\le 3n-1$.

Conversely,  a solution $(u(x), v(x), w(x)) \in \ML(\tilde{T}(x),x^{n},x^{2n}-1; g(x) )\cap \kk[x]_{n-1}^{3}$ implies a solution $(u,v,w)\in \kk^{3\,n}$ of the 
linear system:
$$ 
S 
\, \left(
\begin{array}{c}
u\\
v\\
w\\
\end{array}
\right)
=
\left(
\begin{array}{c}
g\\
0\\
0\\
\end{array}
\right)
$$
where $S$ is has the block structure \eqref{form:S},
so that $T_{2}\, u + w =0$ and $T_{0}\, u - w = (T_{0}+T_{2}) u=g$.
As we have $T_{0}+T_{2}=T$, the vector $u$ is a solution of \eqref{pb:toep},
which ends the proof of the proposition.
\end{proof}
 
\subsection{Euclidean division}

As a consequence of proposition \ref{prop:ML}, we have the following property:
% \begin{remark} \label{rem:unic}
% Let $g(x)\in \kk[x]_{n-1}$. Then $\ML(\Tl(x),$\\$x^{n}, x^{2\,n}-1;g) \cap \kk[x]_{n-1}^{3}$ is either empty or 
% reduced to one element.
% \end{remark}
% If $\ML(\Tl(x), x^{n}, x^{2\,n}-1;g) \cap \kk[x]_{n-1}^{3}$ contains more
% than two elements, then their difference is in $\ML(\Tl(x),$\\$x^{n},
% x^{2\,n}-1) \cap \kk[x]_{n-1}^{3}$ but according to ,
% there is no non-zero syzygy in degree $\le n-1$.

\begin{proposition}\label{divi}
Let $\{(u_1,v_1,w_1),(u_2,v_2,w_2)\}$ a $n$-basis of $\ML(\Tl(x),x^{n},x^{2n}-1)$, the
remainder of the division of 
$\begin{pmatrix}0\\x^n\,g\\g\end{pmatrix}$ by $\begin{pmatrix}
u_1&u_2\\v_1&v_2\\w_1&w_2\end{pmatrix}$ is the vector solution given in the
proposition \eqref{division}. 
\end{proposition}
\begin{proof}
The vector $\begin{pmatrix}0\\x^n\, g\\ -g\end{pmatrix}\in\ML(\Tl(x), x^{n}, x^{2\,n}-1;g)$  (a particular solution). We divide it by $\begin{pmatrix}u_1&u_2\\v_1&v_2\\w_1&w_2\end{pmatrix}$ we obtain
$$\begin{pmatrix}u\\v\\w\end{pmatrix}=\begin{pmatrix}0\\x^n\,g\\g\end{pmatrix}-\begin{pmatrix}u_1&u_2\\v_1&v_2\\w_1&w_2\end{pmatrix}\begin{pmatrix}p\\q\end{pmatrix} $$
$(u,v,w)$ is the remainder of division, thus  $(u,v,w)\in\kk[x]^3_{n-1}\cap\ML(\Tl(x), x^{n}, x^{2\,n}-1;g)$.  However  $(u,v,w)$ is the unique vector $\in\kk[x]^3_{n-1}\cap\ML(\Tl(x), x^{n}, x^{2\,n}-1;g)$ because if there is an other vector then their difference is in $\ML(\Tl(x),x^{n},x^{2\,n}-1) \cap \kk[x]_{n-1}^{3}$ which is equal to $\{(0,0,0)\}$.
\end{proof}

\begin{problem}\label{pb:division}
Given a matrix and a vector of polynomials
$\begin{pmatrix}e(x)&e'(x)\\f(x)&f'(x)\end{pmatrix}$ of degree $n$, and
$\begin{pmatrix}p(x)\\q(x)\end{pmatrix}$ of degree $m\geq n$, such that
$\begin{pmatrix}e_n&e_n'\\f_n&f_n' \end{pmatrix}$ is invertible; find the remainder
of the division of $\begin{pmatrix}p(x)\\q(x)\end{pmatrix} $ by
$\begin{pmatrix}e(x)&e'(x)\\f(x)&f'(x)\end{pmatrix}$.
\end{problem}
\begin{proposition}
The first coordinate of remainder vector of the division of
$\begin{pmatrix}0\\x^ng\end{pmatrix}$ by $\begin{pmatrix}
u&u'\\r&r'\end{pmatrix}$ is the polynomial $v(x)$ solution of
\eqref{pb:toep}.\end{proposition}

We describe here a generalized Euclidean division algorithm to solve
problem \eqref{pb:division}.

Let $E(x)=\begin{pmatrix}p(x)\\q(x)\end{pmatrix}$ of degree $m$, $B(x)=\begin{pmatrix}e(x)&e'(x)\\f(x)&f'(x)\end{pmatrix}$ of degree $n\leq m$.
$E(x)=B(x)Q(x)+R(x)$ with $\deg(R(x))<n,$ and $ \deg(Q(x))\leq m-n$. Let $z=\frac{1}{x}$
 \begin{eqnarray}\label{div}
&E(x)&=B(x)Q(x)+R(x)\nonumber\\
\Leftrightarrow& E(\displaystyle \frac{1}{z})&=B(\frac{1}{z})Q(\frac{1}{z})+R(\frac{1}{z})\nonumber\\
\Leftrightarrow& z^{m}E(\displaystyle \frac{1}{z})&=z^nB(\frac{1}{z})z^{m-n}Q(\frac{1}{z})+z^{m-n+1}z^{n-1}R(\frac{1}{z})\nonumber\\
\Leftrightarrow& \hat{E}(z)&= \hat{B}(z) \hat{Q}(z)+z^{m-n+1} \hat{R}(z)
\end{eqnarray}
with $ \hat{E}(z), \hat{B}(z), \hat{Q}(z), \hat{R}(z)$ are the polynomials obtained by reversing the order of coefficients of polynomials $E(z),B(z),Q(z),R(z)$.
\begin{eqnarray*}
(\ref{div})&\Rightarrow& \frac{ \hat{E}(z)}{ \hat{B}(z)}= \hat{Q}(z)+z^{m+n-1}\frac{  \hat{R}(z)}{ \hat{B}(z)}\\ %\left(\frac{1}{ \hat{B}(z)}=\displaystyle \sum_{i=0}^{\infty}\alpha_i z^i\right)
&\Rightarrow& \hat{Q}(z)=\frac{ \hat{E}(z)}{ \hat{B}(z)} \mod z^{m-n+1}
\end{eqnarray*}
$\displaystyle\frac{1}{ \hat{B}(z)}$ exists  because its coefficient of highest degree is invertible. Thus $\hat{Q}(z)$ is obtained by computing the first $m-n+1$ coefficients of  $\displaystyle\frac{ \hat{E}(z)}{ \hat{B}(z)}$. 

To find $W(x)=\displaystyle\frac{1}{\hat{B}(x)}$ we will use Newton's iteration: Let $f(W)=\hat{B}-W^{-1}$.
\\$f'(W_l).(W_{l+1}-W_l)=-W_l^{-1}(W_l+1-W_l)W_l^{-1}=f(W_l)=\hat{B}-W_l^{-1},$ thus $$W_{l+1}=2W_l-W_l\hat{B}W_l.$$
and $W_0=\hat{B}_0^{-1}$ which exists.
\begin{eqnarray*}
W-W_{l+1}&=&W-2W_l+W_l\hat{B}W_l\\
&=&W(\mathbb{I}_2-\hat{B}W_l)^2\\
&=&(W-W_l)\hat{B}(W-W_l)
\end{eqnarray*}
Thus $W_l(x)=W(x) \mod x^{2l}$ for $l=0,\dots,\lceil\log(m-n+1)
\rceil$. 
\begin{proposition}
We need $\mathcal{O}(n\log(n)\log(m-n)+m\log m)$ arithmetic operations to solve problem \eqref{pb:division}
\end{proposition}
\begin{proof}
We must do $\lceil\log(m-n+1) \rceil$ Newton's iteration
to obtain the first $m-n+1$ coeficients of
$\displaystyle\frac{1}{\hat{B}}=W(x)$. And for each iteration we must do
$\mathcal{O}(n\log n)$ arithmetic operations (multiplication of polynimials of degree $n$). And then we need $\mathcal{O}(m\log m)$ aritmetic operations to do the multiplication $\hat{E}.\displaystyle\frac{1}{\hat{B}}$.
\end{proof}
\subsection{Construction of the generators}
The canonical basis of $\kk[x]^3$ is denoted by $\sigma_1,\sigma_2,\sigma_3$. Let $\rho_1,\,\rho_2$ the generators of $\ML(\Tl(x),x^n,x^{2n}-1)$ of degree $n$ given by 
\begin{equation}\label{base3}
\begin{array}{l}\rho_1=x^n\sigma_1-(u,v,w)=(u_1,v_1,w_1)\\ \rho_2=x^n\sigma_2-(u',v',w')=(u_2,v2_,w_2)\end{array}
\end{equation}
with $(u,v,w),\,(u',v',w')$ are the vector given in  \eqref{base1} and \eqref{base2}.

We will describe here how we compute $(u_1,v_1,w_1)$ and $(u_2,v_2,w_2)$. We will give two methods to compute them, the second one is the method given in \cite{MR1871324}. The first one use the Euclidean gcd algorithm: 

We will recal firstly the algebraic and computational properties of the well known extended Euclidean algorithm (see \cite{MR2001757}):
Given $p(x), p'(x)$ two polynomials in degree $m$ and $m'$ respectively, let
$$\begin{array}{ll}
r_0=p,\qquad&r_1=p',\qquad\\s_0=1,&s_1=0,\\t_0=0,&t_1=1.
\end{array}$$
and define
\begin{eqnarray*}
r_{i+1}&=&r_{i-1}-q_ir_i,\\
s_{i+1}&=&s_{i-1}-q_is_i,\\
t_{i+1}&=&t_{i-1}-q_it_i,
\end{eqnarray*}
where $q_i$ results when the division algorithm is applied to $r_{i-1}$ and $r_i$, i.e. $r_{i-1}=q_ir_i+r_{i+1}$ . $\deg r_{i+1}<\deg r_{i}$ for $i=1,\ldots,l$ with $l$ is such that $r_l=0$, therefore $r_{l-1}=\gcd(p(x),p'(x))$.
\begin{proposition}\label{eea}
The following relations hold:
$$s_ip+t_ip'=r_i\quad \textrm{ and }\quad(s_i,t_i)=1\quad\textrm{ for }i=1,\ldots,l$$
and
$$\left\{\begin{array}{l}\vspace{2mm}
\deg r_{i+1}<\deg r_i, \quad i=1,\ldots,l-1\\ \vspace{2mm}
\deg s_{i+1}>\deg s_i\quad\textrm{ and }\quad \deg t_{i+1}>\deg t_i,\\\vspace{2mm}
\deg s_{i+1}=\deg(q_i.s_i)=\deg v-\deg r_i,\\\vspace{2mm}
\deg t_{i+1}=\deg(q_i.t_i)=\deg u-\deg r_i.
\end{array}\right.$$
\end{proposition}
\begin{proposition}
By applying the Euclidean gcd algorithm in $p(x)=x^{n-1}T$ and $p'(x)=x^{2n-1}$ in degree $n-1$ and $n-2$ we obtain $\rho_1$ and $\rho_2$ respectively
\end{proposition}
\begin{proof}
We saw that $Tu=g$ if and only if there exist $A(x)$ and $B(x)$ such that 
$$\bar{T}(x)u(x)+x^{2n-1}B(x)=x^{n-1}b(x)+A(x)$$ with $\bar{T}(x)=x^{n-1}T(x)$ a polynomial of degree $\leq2n-2$.
In \eqref{base1} and \eqref{base2} we saw that for $g(x)=1$ $(g=e_1)$ and $g(x)=x^nT(x)$ $(g=(0,t_{-n+1},\ldots,t_{-1})^T)$ we obtain a base of $\ML(\Tl(x),x^n,x^{2n}-1)$. $Tu_1=e_1$ if and only if there exist $A_1(x)$, $B_1(x)$ such that
\begin{equation}\label{eea1}\bar{T}(x)u_1(x)+x^{2n-1}B_1(x)=x^{n-1}+A_1(x)\end{equation} 
and $Tu_2=(0,t_{-n+1},\ldots,t_{-1})^T$ 
if and only if there exist $A_2(x)$, $B_2(x)$ such that
\begin{equation}\label{eea2}\bar{T}(x)(u_2(x)+x^{n})+x^{2n-1}B_2(x)=A_2(x)\end{equation} 
with $\deg A_1(x)\leq n-2$ and $\deg A_2(x)\leq n-2$. Thus By applying the extended Euclidean algorithm in $p(x)=x^{n-1}T$ and $p'(x)=x^{2n-1}$ until we have $\deg r_l(x)=n-1$ and $\deg r_{l+1}(x)=n-2$ we obtain 
$$u_1(x)=\frac{1}{c_1}s_l(x),\quad B_1(x)=\frac{1}{c_1}t_l(x),\quad x^{n-1}+A_1(x)=\frac{1}{c_1}r_l(x)$$
and 
$$x^n+u_2(x)=\frac{1}{c_2}s_{l+1}(x),\quad B_2(x)=\frac{1}{c_2}t_{l+1}(x),\quad A_2(x)=\frac{1}{c_2}r_{l+1}(x)$$
with $c_1$ and $c_2$ are the highest coefficients of $r_l(x)$ and $s_{l+1}(x)$ respectively, in fact: 
The equation \eqref{eea1} is equivalent to 
$$
\begin{array}{r}
\overbrace{\phantom{mmmmmmmm}}^{n}\qquad\overbrace{\phantom{mmmmmmm}}^{n-1}\qquad\\
\begin{array}{r}
\left.\begin{array}{l}
{}_{\displaystyle{n-1}}\\\phantom{r}
\end{array}\right\{\\\phantom{r}\\
\left.\begin{array}{l}
\phantom{r}\\n\\\phantom{r}
\end{array}\right\{\\\phantom{r}\\
\left.\begin{array}{l}
{}_{\displaystyle{n-1}}\\\phantom{r}
\end{array}\right\{
\end{array}
\left(
\begin{array}{ccc|ccc}
t_{-n+1}&&&&&\\
\vdots&\ddots&&&&\\
\hline
t_0&\dots&t_{-n+1}&&&\\
\vdots&\ddots&\vdots&&&\\
t_{n-1}&\dots&t_0&&&\\
\hline
&\ddots&\vdots&\;\;1\;\;&&\\
&&&&\;\ddots\;&\\
&&t_{n-1}&&&\;\;1\;\;
\end{array}
\right)
\end{array}
\begin{pmatrix}
\phantom{r}\\u_1\\\phantom{r}\\B_1\\\phantom{r}
\end{pmatrix}
=\begin{pmatrix}\phantom{r}\\A_1\\\phantom{r}\\\hline 1\\0\\\vdots\\0\end{pmatrix}
$$
since $T$ is invertible then the $(2n-1)\times(2n-1)$ block at the bottom is invertible and then $u_1$ and $B_1$ are unique, therefore $u_1,\,B_1$ and $A_1$ are unique. And, by proposition \eqref{eea}, $\deg r_l=n-1$ ($r_l=c_1(x^n+A_1(x)$) then  $\deg s_{l+1}=(2n-1)-(n-1)=n$ and $\deg t_{l+1}=(2n-2)-(n-1)=n-1$ thus, by the same proposition, $\deg s_l\leq n-1$ and $\deg t_l\leq n-2$. Therfore $\frac{1}{c_1} s_l=u_1$ and $\frac{1}{c_1} t_l=B_1$.

Finaly, $Tu=e_1$ if and only if there exist $v(x)$, $w(x)$ such that  
\begin{equation}\Tl(x)u(x)+x^nv(x)+(x^{2n}-1)w(x)=1\end{equation}
$\Tl(x)=T^++x^{2n}T^-=T+(x^{2n}-1)T^-$thus 
\begin{equation}\label{syz}
T(x)u(x)+x^nv(x)+(x^{2n}-1)(w(x)+T^-(x)u(x))=1
\end{equation}
of a other hand $T(x)u(x)-x^{-n+1}A_1(x)+x^nB_1(x)=1$ and $x^{-n+1}A_1(x)=x^n(xA_1)-x^{-n}(x^{2n}-1)xA_1$ thus \begin{equation}\label{syz2}T(x)u(x)+x^{n}(B(x)-xA(x))+(x^{2n}-1)x^{-n+1}A(x)=1\end{equation}
By comparing \eqref{syz} and \eqref{syz2}, and as $1=x^nx^n-(x^{2n}-1)$ we have the proposition and we have $w(x)=x^{-n+1}A(x)-T_-(x)u(x)+1$ which is the part of positif degree of $-T_-(x)u(x)+1$.
\end{proof}

\begin{remark}
A superfast euclidean gcd algorithm, wich uses no more then $\Oc(nlog^2 n)$, is given in \cite{MR2001757} chapter 11.
\end{remark}

The second methode to compute $(u_1,v_1,w_1)$ and $(u_2,v_2,w_2)$ is given in \cite{MR1871324}. We are interested in computing the coefficients of $\sigma_1,\,\sigma_2$, the coefficients of $\sigma_3$ correspond to elements in the ideal $(x^{2n}-1)$ and thus can be obtain by reduction of $(\Tl(x)\,x^n).B(x)$ by $x^{2n}-1$, with $B(x)=\begin{pmatrix}x^n-u_0&-v_0\\-u_1&x^n-v_1\end{pmatrix}=\begin{pmatrix}u(x)&v(x)\\u'(x)&v'(x)\end{pmatrix}$.

A superfast algorithm to compute $B(x)$ is given in \cite{MR1871324}. Let us describe how to compute it.

By evaluation of \eqref{base3} at the roots $\omega_j\in \Unit{2n}$ we deduce that $(u(x)\,v(x))^T$ and $(u'(x)\,v'(x))^T$ are the solution of the following rational interpolation problem:
$$\left\{\begin{array}{l}\Tl(\omega_j)u(\omega_j)+\omega_j^nv(\omega_j)=0\\
\Tl(\omega_j)u'(\omega_j)+\omega_j^nv'(\omega_j)=0\end{array}\right. \textrm{ with } $$
$$\left\{\begin{array}{l}u_n=1,\,v_n=0\\u'_n=0,\,v'_n=1\end{array}\right.$$

\begin{definition}
The $\tau$-degree of a vector polynomial $w(x)=(w_1(x)\,w_2(x))^T$ is defined as 
$$\tau-\deg w(x):=\max\{\deg w_1(x),\,\deg w_2(x)-\tau\}$$
\end{definition}

$B(x)$ is a $n-$reduced basis of the module of all vector polynomials $r(x)\in\kk[x]^2$ that satisfy the interpolation conditions
$$f_j^Tr(\omega_j)=0,\;\;j=0,\ldots,2n-1$$ with $f_j=\begin{pmatrix}\Tl(\omega_j)\\\omega^n_j\end{pmatrix}$.

$B(x)$ is called a  $\tau-$reduced basis (with $\tau=n$) that corresponds to the interpolation data $(\omega_j,f_j),\,j=0,\ldots,2n-1$.
\begin{definition}
A set of vector polynomial in $\kk[x]^{2}$ is called $\tau$-reduced if the $\tau$-highest degree coefficients are lineary independent. 
\end{definition}
\begin{theorem}
Let $\tau=n$. Suppose $J$ is a positive integer. Let $\sigma_1,\ldots,\sigma_J\in\kk$ and $\phi_1,\ldots,\phi_J\in\kk^2$ wich are $\neq(0\,0)^T$. Let $1\leq j\leq J$ and $\tau_J\in\mathbb{Z}$. Suppose that $B_j(x)\in\kk[x]^{2\times2}$ is a $\tau_J$-reduced basis matrix with basis vectors having $\tau_J-$degree $\delta_1$ and $\delta_2$, respectively, corresponding to the interpolation data $\{(\sigma_i,\phi_i); i=1,\ldots,j\}$.

Let $\tau_{j\rightarrow J}:=\delta_1-\delta_2$. Let $B_{j\rightarrow J}(x)$ be a $\tau_{j\rightarrow J}$-reduced basis matrix corresponding to the interpolation data $\{(\sigma_i,B_j^T(\sigma_j)\phi_i); i=j+1,\ldots,J\}$.

Then $B_J(x):=B_j(x)B_{j\rightarrow J}(x)$ is a $\tau_J$-reduced basis matrix corresponding to the interpolation data  $\{(\sigma_i,\phi_i); i=1,\ldots,J\}$.
\end{theorem}
\begin{proof}
For the proof, see \cite{MR1871324}.
\end{proof}

When we apply this theorem for the $\omega_j\in\Unit{2n}$ as interpolation points, we obtain a superfast algorithm ($\mathcal{O}(n\log^2n)$) wich compute $B(x)$.\cite{MR1871324} 

We consider the two following problems: 

\section{Bivariate case}
Let $m\in\mathbb{N},m\in\mathbb{N}$. In this section 
we denote by  $E=\{(i,j);\;0\leq i\leq m-1,\,0\leq j\leq n-1\}$, and $R=\kk[x,y]$. We denote by $\kk[x,y]_{\substack{m\\n}}$ the vector space of bivariate polynomials of degree $\leq m$ in $x$ and $\leq n$ in $y$.
%is a Toeplitz block Toeplitz matrix of order $m \times n$,  $v$ and $b$ are
%vectors of size $m\times n$.
%We write them in the form $v=\begin{pmatrix}v_{\alpha}\end{pmatrix}_{\alpha\in\mathtt{E}}$ and $b=\begin{pmatrix}b_{\alpha}\end{pmatrix}_{\alpha\in\mathtt{E}}$
\begin{notation}
For a block matrix $M$, of block size $n$ and each block is of size $m$, we will use the following indication :
\begin{equation}
M=\left(M_{(i_1,i_2),(j_1,j_2)}\right)_{\substack{0\leq i_1,j_1\leq m-1\\0\leq i_2,j_2\leq n-1}}=(M_{\alpha\beta})_{\alpha,\beta\in E}.
\end{equation}
$(i_2,j_2)$ gives the block's positions, $(i_1,j_1)$ the position in the blocks.

\end{notation}

\begin{problem}
Given a Toeplitz block Toeplitz matrix $T=(t_{\alpha-\beta})_{\alpha\in E,\beta\in E}\in\kk^{mn\times mn}$ ($T=(T_{\alpha\beta})_{\alpha,\beta\in E}$ with $T_{\alpha\beta}=t_{\alpha-\beta}$) of size $mn$ and $g=(g_{\alpha})_{\alpha\in\mathtt{E}}\in\kk^{mn}$, find $u=(u_{\alpha})_{\alpha\in\mathtt{E}}$ such that 
\begin{equation}\label{pb:toblto}
T\, u=g
\end{equation}
\end{problem}
% This problem is equivalent to the following one:
% \begin{problem}
% Given a Hankel block Hankel matrix $H=(h_{\alpha+\beta})_{\alpha\in E,\beta\in E}\in\kk^{mn\times mn}$ of size $mn$ and $g=(g_{\alpha})_{\alpha\in\mathtt{E}}\in\kk^{mn}$, find $w=(w_{\alpha})_{\alpha\in\mathtt{E}}$ such that 
% \begin{equation}\label{pb:hablha}
% H\,w= g
% \end{equation}
% \end{problem}
% Again we take $H=TJ$ and $w=Ju$ where $J$ is that matrix of size $m\,n$, with
% $1$ on the antidiagonal and $0$ else where.

\begin{definition} 
We define the following polynomials:
\begin{itemize}
\item $T(x,y):=\displaystyle\sum_{(i,j)\in \mathtt{E}-\mathtt{E}}t_{i,j}x^iy^j$, 
\item $\tilde{T}(x,y):=\displaystyle\sum_{i,j=0}^{2n-1,2m-1}\tilde{t}_{i,j}x^iy^j$ with\\ 
$\tilde{t}_{i,j}:=\left\{\begin{array}{ll}t_{i,j}&\textrm{si }i<m, j<n\\ 
t_{i-2m,j}&\textrm{si }i\geq m, j<n\\
t_{i,j-2n}&\textrm{si }i<m, j\geq n\\
t_{i-2m,j-2n}&\textrm{si }i\geq m,i\geq n \end{array}\right.$,
\item $u(x,y):=\displaystyle\sum_{(i,j)\in
\mathtt{E}}u_{i,j}\, x^iy^j$, 
$g(x,y):=\displaystyle\sum_{(i,j)\in \mathtt{E}}g_{i,j}
x^{i} y^{j}$.
\end{itemize} 
\end{definition}

\subsection{Moving hyperplanes}
For any non-zero vector of polynomials $\ab=(a_{1},\ldots,a_{n})\in
\kk[x,y]^{n}$, we denote by $\mathcal{L}(\ab)$  
the set of vectors $(h_{1},\ldots,h_{n})\in\kk[x,y]^n$ 
such that
\begin{equation} \label{eq:mvplanes}
\sum_{i=1}^{n} a_{i} \, h_{i} = 0.
\end{equation}
It is a $\kk[x,y]$-module of $\kk[x,y]^n$. 

\begin{proposition}\label{division2}
The vector $u$ is solution of \eqref{pb:toblto}  if and only if  there exist
$h_{2}, \ldots, h_{9} \in \kk[x,y]_{\substack{m-1\\n-1}}$ such that 
$(u(x,y),h_{2}(x,y),\ldots,h_{9}(x,y))$ belongs to\\
$\ML(\Tl(x,y),x^{m}, x^{2\, m} -1, y^{n}, x^{m}\,y^{n},(x^{2\,m}-1)\, y^{n},
y^{2\,n}-1,
x^{m} (y^{2\,n}-1),
(x^{2\,m}-1)\, (y^{2\,n}-1)).
$
\end{proposition}
\begin{proof}
Let $L=\{x^{\alpha_{1}}y^{\alpha_{2}}, 0\le \alpha_{1} \le m-1, 0 \le
\alpha_{2} \le n-1\}$, and  $\Pi_{E}$ the projection of $R$ on the vector space
generated by $L$.   
By \cite{MR1762401}, we have
\begin{equation}\label{dv}
T\,u=g \Leftrightarrow \Pi_{E}(T(x,y)\,u(x,y))=g(x,y)
\end{equation}
which implies that
\begin{eqnarray} \label{eq:sol1}
T(x,y)u(x,y) & = & g(x,y)+x^{m}y^{n}A_1(x,y)+x^{m}y^{-n}A_2(x,y)+x^{-m}y^{n}A_3(x,y)+ x^{-m}y^{-n}A_4(x,y) \nonumber\\
& + &  x^{m}A_5(x,y)+x^{-m} A_6(x,y)+y^{n}A_7(x,y)+y^{-n}A_8(x,y),
\end{eqnarray} 
where the $A_i(x,y)$ are  polynomials of degree at most $m-1$ in $x$ and $n-1$ in $y$.
Since $\omega^{m} = \omega^{-m}$, $\upsilon^{n}=\upsilon^{-n}$,
$\Tl(\omega,\upsilon)=T(\omega,\upsilon)$ for $\omega\in \Unit{2\,m}$,
$\upsilon\in \Unit{2\,n}$, we deduce by evaluation at the roots $\omega\in
\Unit{2\,m}$, $\upsilon\in \Unit{2\,n}$ that
\begin{equation*}
R(x,y):= \Tl(x,y)u(x,y)+x^{m} h_{2}(x,y)+ y^{n}h_{4}(x,y)+ x^{m}y^{n} h_{5}(x,y) -  g(x,y) \in (x^{2\,m}-1,y^{2\,n}-1)
\end{equation*}
with $h_{2}=-(A_{5}+A_{6})$, $h_{4}=-(A_{7}+A_{8})$, 
$h_{5}=-(A_{1}(x,y)+A_{2}(x,y)+A_{3}(x,y)+A_{4}(x,y))$. 

By reduction by the polynomials $x^{2\,m}-1$, $y^{2\, n}-1$, and as $R(x,y)$
is of degree $\le 3m -1$ in $x$ and $\le 3n-1$ in $y$, 
there exist $h_{3}(x,y), h_{6}(x,y),\ldots,h_8(x,y) \in\kk[x,y]_{\substack{m-1\\n-1}}$ such that 
\begin{eqnarray}\label{reduc1}
&&\Tl(x,y)u(x,y)+x^m\, h_2(x,y)+(x^{2m}-1)h_{3}(x,y)+y^n h_{4}(x,y)+ x^my^nh_{5}(x,y)+\\\nonumber
&&(x^{2m}-1)y^nh_{6}(x,y)+(y^{2n}-1)h_{7}(x,y)+x^m(y^{2m}-1)h_{7}(x,y)+(x^{2n}-1)(y^{2n}-1)h_{8}(x,y)=g(x,y).
\end{eqnarray}
Conversely a solution of \eqref{reduc1} can be transformed into a solution of 
\eqref{eq:sol1}, which ends the proof of the proposition.
\end{proof}

In the following, we are going to denote by $\Tv$ the vector $\Tv=(\Tl(x,y),x^{m}, x^{2\, m} -1, y^{n}, x^{m}\,y^{n},(x^{2\,m}-1)\, y^{n},y^{2\,n}-1,x^{m} (y^{2\,n}-1),(x^{2\,m}-1)\, (y^{2\,n}-1))$.

\begin{proposition}\label{syzygy:2d:deg}
There is no elements of $\kk[x,y]_{\substack{m-1\\n-1}}$ in $\ML(\Tv)$. 
\end{proposition}
\begin{proof}
We consider the map
\begin{eqnarray}\label{syzfunction2}
\kk[x,y]_{\substack{m-1\\n-1}}^9&\rightarrow & \kk[x,y]_{\substack{3m-1\\3n-1}}\\
p(x,y)=(p_1(x,y),\ldots,p_9(x,y))  & \mapsto & \Tv.p \\
\end{eqnarray}
which $9mn \times 9mn$ matrix is of the form 
\begin{equation}S:=\left(\begin{array}{c|c|c|c}
\begin{matrix}
\\
T_0\\
\end{matrix}
&
\begin{matrix}
E_{21}&-E_{11}+E_{31}\\
\vdots&\vdots\\
E_{2n}&-E_{1n}+E_{3n}
\end{matrix}
&
&
\begin{matrix}
-E_{11}&-E_{21}&E_{11}-E_{31}\\
\vdots&\vdots&\vdots\\
-E_{1n}&-E_{2n}&E_{1n}-E_{3n}
\end{matrix}
\\\hline
\begin{matrix}
\\
T_1\\
\end{matrix}
&
&
\begin{matrix}
E_{11}&E_{21}&-E_{11}+E_{31}\\
\vdots&\vdots&\vdots\\
E_{1n}&E_{2n}&-E_{1n}+E_{3n}
\end{matrix}
&
\\\hline
\begin{matrix}
\\
T_2\\
\end{matrix}
&&&
\begin{matrix}
E_{11}&E_{21}&-E_{11}+E_{31}\\
\vdots&\vdots&\vdots\\
E_{1n}&E_{2n}&-E_{1n}+E_{3n}
\end{matrix}
\end{array}\right)
\end{equation}
with $E_{ij}$ is the $3m\times mn$ matrix $e_{ij}\otimes I_m$ and $e_{ij}$ is the $3\times n$ matrix with entries equal zero except the $(i,j)$th entrie equal $1$. And the matrix $\begin{pmatrix}T_0\\T_1\\T_2 \end{pmatrix}$ is the following $9mn\times m$ matrix 
$$
\begin{pmatrix}
t_0&0&\ldots&0\\t_1&t_0&\ldots&0\\
\vdots&\ddots&\ddots&\vdots\\
t_{n-1}&\ldots&t_1&t_0\\
\hline
0&t_{n-1}&\ldots&t_1\\t_{-n+1}&0&\ddots&\vdots\\
\vdots&\ddots&\ddots&t_{-n+1}\\
t_{-1}&\ldots&t_{-n+1}&0\\
\hline
0&t_{-1}&\ldots&t_{-n+1}\\\vdots&\ddots&\ddots&\vdots\\
\vdots&&\ddots&t_{-1}\\
0&\ldots&\ldots&0\\
\end{pmatrix}
\textrm{ and }
t_i=\begin{pmatrix}
t_{i,0}&0&\ldots&0\\t_{i,1}&t_{i,0}&\ldots&0\\
\vdots&\ddots&\ddots&\vdots\\
t_{i,n-1}&\ldots&t_{i,1}&t_{i,0}\\
\hline
0&t_{i,m-1}&\ldots&t_{i,1}\\t_{i,-m+1}&0&\ddots&\vdots\\
\vdots&\ddots&\ddots&t_{i,-m+1}\\
t_{i,-1}&\ldots&t_{i,-m+1}&0\\
\hline
0&t_{i,-1}&\ldots&t_{i,-m+1}\\\vdots&\ddots&\ddots&\vdots\\
\vdots&&\ddots&t_{i,-1}\\
0&\ldots&\ldots&0\\
\end{pmatrix}
$$
For the same reasons in the proof of proposition \eqref{prop:ML} the matrix $S$ is invertible.
\end{proof}

\begin{theorem}\label{free:mod}
For any non-zero vector of polynomials $\ab=(a_{i})_{i=1,\ldots, n}\in \kk[x,y]^{n}$, the
$\kk[x,y]$-module $\ML(a_{1},\ldots, a_{n})$ is free of rank $n-1$.
\end{theorem}
\begin{proof}
Consider first the case where $a_{i}$ are monomials. 

$a_{i}= x^{\alpha_{i}}y^{\beta_{i}}$ 
that are sorted in lexicographic order such that $x<y$, $a_{1}$ being the
biggest and $a_{n}$ the smallest. 
Then the module of syzygies of $\ab$ is generated by 
the $S$-polynomials:
$$ 
S(a_{i},a_{j}) = \lcm(a_{i},a_{j})({\sigma_{i}\over a_{i}}- {\sigma_{j}\over
a_{j}}),
$$
where $(\sigma_{i})_{i=1,\ldots,n}$ is the canonical basis of
$\kk[x,y]^{n}$ \cite{MR1322960}. We easily check that 
$S(a_{i},a_{k}) = 
  {\lcm(a_{i},a_{k})\over\lcm(a_{i},a_{j})} S(a_{i},a_{j}) 
- {\lcm(a_{i},a_{k})\over\lcm(a_{j},a_{k})} S(a_{j},a_{k})$
if $i\neq j \neq k$ and $\lcm (a_{i},a_{j})$ divides $\lcm(a_{i},a_{k})$. 
Therefore $\ML(\ab)$ is generated by the $S(a_{i},a_{j})$ which are minimal
for the division, that is, by $S(a_{i},a_{i+1})$ (for $i=1,\ldots,n-1$),
since the monomials $a_{i}$ are sorted lexicographically. As the syzygies
$S(a_{i},a_{i+1})$ involve the basis elements $\sigma_{i},\sigma_{i+1}$, 
they are linearly independent over $\kk[x,y]$, which shows that 
$\ML(\ab)$ is a free module of rank $n-1$ and that we have the following
resolution:
$$ 
0 \rightarrow \kk[x,y]^{n-1} \rightarrow \kk[x,y]^{n} \rightarrow
(\ab) \rightarrow 0.
$$

Suppose now that $a_{i}$ are general polynomials $\in \kk[x,y]$ and let us
compute a Gr\"obner basis of $a_{i}$, for a monomial ordering refining the 
degree \cite{MR1322960}. 
We denote by $m_{1},\ldots, m_{s}$ the leading terms of the polynomials in 
this Gr\"obner basis, sorted by lexicographic order.

The previous construction yields a resolution of $(m_{1}$, $\ldots$, $m_{s})$:
$$ 
0 \rightarrow \kk[x,y]^{s-1} \rightarrow \kk[x,y]^{s} \rightarrow 
(m_{i})_{i=1,\ldots,s} \rightarrow 0.
$$
Using \cite{MR839576} (or \cite{MR1322960}), 
this resolution can be deformed into a resolution 
of $(\ab)$, of the form
$$ 
 0 \rightarrow \kk[x,y]^{p} \rightarrow \kk[x,y]^{n} \rightarrow 
(\ab) \rightarrow 0,
$$
which shows that $\ML(\ab)$ is also a free module. Its rank $p$
is necessarily equal to $n-1$, since the alternate sum of the dimensions
of the vector spaces of elements of degree $\le \nu$  in each module of this
resolution should be $0$, for $\nu \in \NN$.
\end{proof}

\subsection{Generators and reduction}
In this section, we describe an explicit set of generators of $\ML(\Tv)$.
The canonical basis of $\kk[x,y]^9$ is denoted by $\sigma_1,\ldots, \sigma_9$.

First as $\Tl(x,y)$ is of degree $\le 2\,m-1$ in $x$ and $\le 2\, n-1$ in $y$ and as the function \eqref{syzfunction2} in surjective, there
exists $u_1, u_2 \in \kk[x,y]_{\substack{m-1\\n-1}}^9$ such that 
$\Tv\cdot u_1 = \Tl(x,y) x^m$, $\Tv\cdot u_2 = \Tl(x,y) y^n$. Thus,
$$
\begin{array}{l}
\rho_{1} = x^m \sigma_1 - u_1 \in \ML(\Tv),\\
\rho_{2} = y^n \sigma_1 - u_2 \in \ML(\Tv).
\end{array} 
$$
We also have $u_3 \in \kk[x,y]_{\substack{m-1\\n-1}}$, such that
$\Tv\cdot u_3 = 1 = x^m \, x^m - (x^{2\,m}-1) =  y^n\, y^n - (y^{2\,n} -1)$. We deduce that
$$
\begin{array}{l}
\rho_{3} = x^m \sigma_2 - \sigma_3 -u_3  \in \ML(\Tv),\\
\rho_{4} = y^n \sigma_4 - \sigma_7 -u_3  \in \ML(\Tv).
\end{array} 
$$
Finally, we have the obvious relations:
$$
\begin{array}{l}
\rho_{5} = y^n \sigma_2 - \sigma_5  \in \ML(\Tv),\\
\rho_{6} = x^m \sigma_4 - \sigma_5  \in \ML(\Tv), \\
\rho_{7} = x^m \sigma_5 - \sigma_6  + \sigma_4 \in \ML(\Tv), \\
\rho_{8} = y^n \sigma_5 - \sigma_8  + \sigma_2 \in \ML(\Tv). \\
\end{array} 
$$
\begin{proposition}\label{basis}
The relations $\rho_{1},\ldots,\rho_{8}$ form a basis of $\ML(\Tv)$.
\end{proposition}
\begin{proof}
Let $\hb=(h_1,\ldots,h_9)\in \ML(\Tv)$.
By reduction by the previous elements of $\ML(\Tv)$, we can assume that
the coefficients $h_{1}, h_{2}, h_{4}, h_{5}$ are in $\kk[x,y]_{\substack{m-1\\n-1}}$.
Thus, $\Tl(x,y) h_{1} + x^{m} h_{2} + y^{n} h_{4} + x^{m} y^{n} h_{5} \in (x^{2\,n}-1,y^{2\,m}-1)$.
As this polynomial is of degree $\le 3\,m-1$ in $x$ and $\le 3\, n-1$ in $y$, by reduction by 
the polynomials, we deduce that the coefficients $h_{3},h_{6},\ldots,h_{9}$ are in $\kk[x,y]_{\substack{m-1\\n-1}}$.
By proposition \ref{syzygy:2d:deg}, there is no non-zero syzygy in $\kk[x,y]_{\substack{m-1\\n-1}}^{9}$. Thus we have 
$\hb=0$ and every element of $\ML(\Tv)$ can be reduced to $0$ by the previous relations.
In other words, $\rho_{1},\ldots, \rho_{8}$ is a generating set of 
the $\kk[x,y]$-module $\ML(\Tv)$. By theorem \ref{free:mod}, the relations
$\rho_{i}$ cannot be dependent over $\kk[x,y]$ and thus form a basis of
$\ML(\Tv)$. 
\end{proof}

\subsection{Interpolation}
Our aim is now to compute efficiently a system of generators of $\ML(\Tv)$. 

More precisely, we are interested in computing the coefficients of
$\sigma_{1},\sigma_{2},\sigma_{4}, \sigma_{5}$ of
$\rho_{1},\rho_{2},\rho_{3}$. Let us call $B(x,y)$ the corresponding
coefficient matrix, which is of the form: 
\begin{equation}\label{form:B}
%B(x,y)  =
\left( \begin{array}{cccc}
x^m & y^n & 0    \\
0  &  0  & x^{m} \\
0  &  0  & 0     \\
0  &  0  & 0     \\
\end{array}
\right)
+
\kk[x,y]_{\substack{m-1\\n-1}}^{4,3}
\end{equation}
Notice that the other coefficients of the relations
$\rho_{1},\rho_{2},\rho_{3}$ correspond to elements in the ideal
$(x^{2\,m}-1,y^{2\,n}-1)$ and thus can be obtained easily by reduction of the
entries of $(\Tl(x,y),x^m,y^n,x^m\,y^n)\cdot B(x,y)$ by the polynomials
$x^{2\,m}-1,y^{2\,n}-1$.

Notice also that the relation $\rho_{4}$ can be easily deduced from
$\rho_{3}$, since we have $\rho_{3} -x^{m}\sigma_{2} +\sigma_{3}+y^n\,
\sigma_{4} -\sigma_{7} = \rho_{4}$.  Since the other relations $\rho_{i}$
(for $i>4$) are explicit and independent of $\Tl(x,y)$, we can easily deduce
a basis of $\ML(\Tv)$ from the matrix $B(x,y)$.

As in $\ML(\Tv)\cap\kk[x,y]_{\substack{m-1\\n-1}}$ there is only one element, thus by computing the basis given in proposition \eqref{basis} and reducing it we can obtain this element in $\ML(\Tv)\cap\kk[x,y]_{\substack{m-1\\n-1}}$ which gives us the solution of $Tu=g$. We can give a fast algorithm to do these two step, but a superfast algorithm is not available.

\section{Conclusions}
We show in this paper a correlation between the solution of a Toeplitz system and the syzygies of polynomials. We generalized this way, and we gave a correlation between the solution of a Toeplitz-block-Toeplitz system and the syzygies of bivariate polynomials. In the univariate case we could exploit this correlation to give a superfast resolution algorithm. The generalization of this technique to the bivariate case is not very clear and it remains an important challenge.

\bibliographystyle{abbrv}
\bibliography{matrstr}
\end{document}